\documentclass{amsart}

\usepackage{esvect} 
\usepackage{tikz}
\usepackage{tikz-cd}

\newcommand{\myholim}[1]{\mathbin{\operatorname*{holim\ }_{#1}^{}}}

\newcommand{\myhocolim}[1]{\mathbin{\operatorname*{hocolim}_{#1}^{}}}

\DeclareMathOperator{\Fun}{Fun}
\DeclareMathOperator{\cof}{cof}
\DeclareMathOperator{\id}{id}
\DeclareMathOperator{\gr}{gr}

\DeclareMathOperator{\Map}{Map}
\DeclareMathOperator{\Res}{Res}

\newcommand{\DD}{\mathbb{D}}
\DeclareMathOperator{\Ind}{Ind}

\usepackage{amsthm}

    \theoremstyle{plain}
    \newtheorem{theorem}{Theorem}[section]
    \newtheoremstyle{TheoremNum}
        {}{}              %%% space between body and thm
        {\itshape}                      %%% Thm body font
        {}                              %%% Indent amount (empty = no indent)
        {\bfseries}                     %%% Thm head font
        {.}                             %%% Punctuation after thm head
        { }                             %%% Space after thm head
        {\thmname{#1}\thmnote{ \bfseries #3}}%%% Thm head spec
    \theoremstyle{TheoremNum}

    \newtheoremstyle{LemmaNum}
        {}{}              %%% space between body and thm
        {\itshape}                      %%% Thm body font
        {}                              %%% Indent amount (empty = no indent)
        {\bfseries}                     %%% Thm head font
        {.}                             %%% Punctuation after thm head
        { }                             %%% Space after thm head
        {\thmname{#1}\thmnote{ \bfseries #3}}%%% Thm head spec
    \theoremstyle{TheoremNum}

\theoremstyle{plain}

\newtheorem{lemma}[theorem]{Lemma}
\newtheorem{proposition}[theorem]{Proposition}

\newtheorem{corollary}[theorem]{Corollary}

\newtheorem*{proposition*}{Proposition}
\newtheorem*{corollary*}{Corollary}

\theoremstyle{definition}
\newtheorem{definition}[theorem]{Definition}

\newtheorem{remark}[theorem]{Remark}

\begin{document}

% Title of document, usually lower case except for first word
% and proper nouns.  Avoid unnecessary symbols.
\title{The $v_n$-periodic Goodwillie tower on Wedges and Cofibres}

% First author's postal address, with *line breaks* like below.  Do
% not use \\ to separate lines.  It will appear on one line in the
% article, but will be used exactly as typed below to send a
% complimentary copy of the journal, so ensure that it is a complete
% postal address with line breaks as would appear on an envelope.
%
% Also, if it is not obvious, please add a comment with % that says
% what part is a district within a city, what part is a city, 
% what part is a region or province and what part is a postal code.
%
% We recommend using the most current address possible.  
% If you want the journal copy sent to a different address than
% the one below, please let Dan Christensen <jdc@uwo.ca> know.

\author[Lukas Brantner]{Lukas Brantner}
\address{Harvard University, Department of Mathematics, 1 Oxford Street, 02138 Cambridge, Massachusetts, USA}
\email{brantner@maths.ox.ac.uk}
\curraddr{Merton College,  
         Oxford University,
Oxford, OX1 4JD, 
United Kingdom}

\author[Gijs Heuts]{Gijs Heuts}
\address{Harvard University, Department of Mathematics, 1 Oxford Street, 02138 Cambridge, Massachusetts, USA}
\email{g.s.k.s.heuts@uu.nl}
\curraddr{Mathematical Institute,
         Utrecht University,
         Budapestlaan 6,
         Utrecht, 3584 CD,
         The Netherlands}

% Additional authors done in the same way.

% If the author names are too long for the running head, use
% the following command to specify a shorter version:
%\shortauthors{DOE, SMITH \andname\ WILLIAMS}

% AMS 2010 Mathematics Subject Classification.  List one or several,
% separated by commas, ending in a period.
% See http://www.ams.org/mathscinet/msc/msc2010.html for the 2010
% numbering system. 
% Use \classification[2000]{12X34, 55X78.} if you must use codes
% from the 2000 numbering system.

% Keywords of the article, usually singular, no leading caps.  
% Separated by commas, ending with period. 
% Abstract comes before maketitle
\begin{abstract}
We introduce general methods to analyse the Goodwillie tower of the identity functor on a wedge $X \vee Y$ of spaces  (using the Hilton--Milnor theorem) and on the cofibre $\cof(f)$ of a map $f\colon X \rightarrow Y$. We deduce some consequences for $v_n$-periodic homotopy groups: whereas the Goodwillie tower is finite and converges in periodic homotopy when evaluated on spheres (Arone--Mahowald),  we show that neither of these statements remains true for wedges and Moore spaces.\vspace{-8pt}
\end{abstract}

\maketitle

% Text of Document.  Use constructs such as \section, \subsection,
% \begin{theorem} ... \end{theorem}, \begin{proof} ... \end{proof}, etc.

\section{Introduction and main results} 
Recently, Behrens and Rezk \cite{BR} have provided a new perspective on the calculation of $v_h$-periodic homotopy groups of spheres by relating the Bousfield--Kuhn functor to topological Andr\'{e}-Quillen cohomology. Their result applies to the class of spaces for which the \emph{$v_h$-periodic Goodwillie tower converges} (see Definition \ref{definition:convergence}). This naturally raised the question of which spaces are contained in this class. In 1998, Arone and Mahowald \cite{2} established that spheres satisfy the desired hypothesis (see Theorem \ref{thm:AM} below), but beyond this, knowledge was scarce.

In this paper, we introduce new methods for the analysis of the Goodwillie tower on a wedge of spaces or on the cofibre of a map.
We use them to prove that several natural classes of spaces have \textit{divergent} $v_h$-periodic Goodwillie towers. 
In particular, the $v_h$-periodic homotopy groups of these spaces cannot be completely recovered from the TAQ-based Lie algebra models of Behrens and Rezk (cf. \cite{BR2}).
Our methods in fact provide a useful tool in the study of $v_h$-periodic spaces through spectral Lie algebras.
For example, Theorem \ref{thm:HMGoodwillie} was originally used in \cite{heuts} to establish a more direct relation between the Bousfield--Kuhn functor and spectral Lie algebras which does not suffer from the aforementioned convergence issues.

Given a pointed space $X$, we can evaluate the \emph{Goodwillie tower of the identity} \mbox{on $X$} (see \cite{1}) and obtain a diagram
\[ X \rightarrow \cdots \rightarrow P_2 X \rightarrow P_1 X. \]
The first space $P_1 X$ is simply given by $QX = \Omega^\infty\Sigma^\infty X$, and for $n \geq 2$, the $n^{\mathrm{th}}$ \textit{Goodwillie layer} is the homotopy fibre $D_n X$ of the map $P_n X \rightarrow P_{n-1}X$. By work of Johnson and Arone--Mahowald (cf. \cite{Joh}, \cite{2}), it can be expressed as the infinite loop space of the spectrum  
\[\mathbf{D}_n X = (\partial_n \mathrm{id} \wedge X^{\wedge n})_{h\Sigma_n}\] where $\partial_n \mathrm{id}$ is the Spanier-Whitehead dual of the suspended \emph{partition complex}. For simply-connected (and more generally nilpotent) spaces $X$, the natural map 
\begin{equation*}
X \longrightarrow \myholim{n}  P_n X
\end{equation*}
is known to be an equivalence -- we say the Goodwillie tower \emph{converges on $X$}.

Recall that for pointed connected spaces $X_1, \ldots, X_k$, the Hilton--Milnor theorem asserts the existence of a weak equivalence
\begin{equation*}
\sideset{}{'}\prod_{w \in \mathbf{L}_k} \Omega \Sigma \bigl( w(X_1,\ldots, X_k)\bigr) \xrightarrow{ \phantom{abefgh} H \phantom{abcfgh} } \Omega\Sigma (X_1 \vee \cdots \vee  X_k).
\end{equation*}
Here $\mathbf{L}_k$ denotes an ordered set of Lie words $w$ forming a basis for the free (ungraded) Lie algebra on $k$  generators $x_1, \ldots, x_k$. 
We evaluate such a word on a $k-$tuple of spaces $X_1,\ldots,X_k$ by letting the bracket act as a smash - for example $[x_1,[x_2,x_2]](X_1,X_2,X_3) = X_1 \wedge X_2 \wedge X_2$. Finally, $\prod '$ denotes the \emph{weak infinite product}, i.e. the filtered colimit of products indexed over the finite subsets of $\mathbf{L}_k$.

Unstably, wedge sums of spaces are often harder to understand than smash products. The general strategy is to use the equivalence $H$ to compute invariants of the wedge sum on the right in terms of the corresponding invariants of the various smash products on the left.
This technique has been employed to express the Goodwillie \textit{layers} $D_n$ evaluated on a wedge sum in terms of the layers $D_n$ evaluated on various related smash products:

\begin{theorem}[Arone--Kankaanrinta,  cf. Theorem 0.1. in \cite{3}]
For any collection of pointed connected spaces $X_1,\ldots, X_k$, there is an equivalence
\[ \prod_{\sum n_i = n}  \ \ \ \  \sideset{}{}\prod_{\substack{d|\gcd(n_i)\\ w\in B(\frac{n_1}{d}, \dots , \frac{n_k}{d})}} \ \ \ \Omega D_d\left( \Sigma w(X_1 , \ldots, X_k)\right) \xrightarrow{\sim}\Omega   D_n( \Sigma X_1 \vee \cdots \vee\Sigma X_k). 
\]
Here $B(\frac{n_1}{d}, \dots , \frac{n_k}{d})$ is the set of words in the basis $\mathbf{L}_k$ involving the $i^{th}$ letter $\frac{n_i}{d}$ times.
\end{theorem}
We lift this result concerning the layers $D_n$ to the level of towers. Essentially, we will show that each component
\[
\Omega \Sigma \bigl( w(X_1,\ldots, X_k)\bigr) \xrightarrow{ \phantom{abefgh} H_w \phantom{abcfgh} } \Omega\Sigma (X_1 \vee \cdots \vee  X_k)
\]
of the Hilton--Milnor map can be refined to a map of towers
\[
\Omega P_{n}(\Sigma w(X_1,\ldots, X_k)) \xrightarrow{ \phantom{abefgh} H_{w,n} \phantom{abcfgh} } \Omega P_{n|w|}(\Sigma(X_1 \vee \cdots \vee X_k)),
\]
i.e., a sequence of maps as written which are natural in $n$. Here $|w|$ denotes the length of a word $w$. We use these maps to assemble the Goodwillie tower on wedge sums as follows:

\begin{theorem}[\ref{thm:HMGoodwillie}]
Fix pointed connected spaces $X_1, \ldots, X_k$.
The Hilton--Milnor map refines to an equivalence of towers:
\[
\sideset{}{'}\prod_{w\in \mathbf{L}_k}\Omega P_{\lfloor \frac{n}{|w|}\rfloor} \Bigl(\Sigma w(X_1, \ldots, X_k) \Bigr) \xrightarrow{\sim} \Omega P_n (\Sigma X_1 \vee \cdots \vee \Sigma X_k).
\]
\end{theorem}

We prove this result by extending the known comparison results between single- and multivariable Goodwillie calculus in Section \ref{sec:calculus}. 

Let $\mathcal{S}_*$ denote the category of pointed spaces and consider a functor of $k$ variables 
\begin{equation*}
F: \mathcal{S}_*^{\times k} \longrightarrow \mathcal{S}_*
\end{equation*}
which preserves weak equivalences and filtered colimits. We can apply (single-variable) Goodwillie calculus to such a functor to obtain $n$-excisive approximations $P_n F$ for all $n \geq 0$. However, we can also apply multivariable calculus and form approximations `in each variable separately'. Given a tuple $(n_1, \ldots, n_k)$ of positive integers, we may form a Goodwillie approximation $P_{\vv n}:=P_{n_1, \ldots, n_k}F$ which is $n_i$-excisive in the $i^{th}$ variable.
The key Lemma \ref{towersonwedges}, which is of some independent interest, allows us to compute single-variable Goodwillie towers on wedges in terms of associated multi-variable towers.

We apply Theorem \ref{thm:HMGoodwillie} to $v_h$-periodic homotopy theory. First we recall the basic setup. Fix a prime $p$ and a natural number $h$. Everything that follows will implicitly be localised at $p$. Recall that a \emph{$v_h$ self-map} of a finite pointed space $V$ is a map $v: \Sigma^d V \rightarrow V$ which induces an isomorphism in $K(h)_*$ and is nilpotent in $K(i)_*$ for $i \neq h$. Here $K(i)_*$ denotes the (reduced) $i^{th}$ Morava $K$-theory at the prime $p$. A finite pointed space $V$ is \emph{of type $h$} if $K(i)_* V = 0$ for $i < h$ and $K(h)_*V \neq 0$. Mitchell \cite{mitchell} established that finite type $h$ spaces exist for every $h \geq 0$. The periodicity theorem of Hopkins-Smith (see Theorem $9$ of \cite{6}) guarantees that every type $h$ space admits a $v_h$ self-map after suspending it sufficiently many times.

From now on, we fix a finite pointed space $V$ of type $h$ together with a $v_h$-self map $v: \Sigma^d V \rightarrow V$. For any pointed space $X$, one can define the \emph{$v_h$-periodic homotopy groups} of $X$ with coefficients in $V$ by considering the homotopy groups of the space
$\mathrm{Map}_*(V, X)$
and inverting the action of $v$ by precomposition. More precisely, one considers the system of spaces
\begin{equation*}
\mathrm{Map}_*(V,X) \xrightarrow{v^*} \mathrm{Map}_*(\Sigma^d V,X) \xrightarrow{v^*} \mathrm{Map}_*(\Sigma^{2d}V,X) \longrightarrow \cdots.
\end{equation*}
One can think of this sequence as defining a spectrum $\Phi_v X$ with constituent spaces $(\Phi_v X)_{kd} = \mathrm{Map}_*(V,X)$ for every $k \geq 0$ and with structure maps
\begin{equation*}
v^*: \mathrm{Map}_*(V,X) \rightarrow \Omega^d\mathrm{Map}_*(V,X).
\end{equation*}
The $v_h$-periodic homotopy groups of $X$ with coefficients in $V$ are then the homotopy groups of this spectrum $\Phi_v X$. This functor $\Phi_v$ is called the \emph{telescopic functor} associated to $v$. Up to equivalence, %comma
the value of $\Phi_v X$ is independent of the choice of self-map $v$ by Corollary 3.7 of \cite{6}, although it does still depend on the chosen space $V$. Observe that the homotopy groups of $\Phi_v$ are periodic with period $d$. Also, it is clear that $\Phi_v$ preserves filtered colimits and finite homotopy limits. It is possible to take a certain homotopy limit over choices of coefficient complexes $V$ and thus define a functor $\Phi$, the \emph{Bousfield--Kuhn functor}, which is independent of any choices (see \cite{kuhn} for a comprehensive overview, or \cite{kuhnmorava} for a more original reference). % add a reference to Bousfield

It is useful to know that the functor $\Phi_v$ takes values in the category $\mathrm{Sp}_{T(h)}$ of \emph{$T(h)$-local spectra} (see Theorem 4.2 of \cite{kuhn}). Here $T(h)$ is the telescope of a $v_h$ self-map on a finite $p$-local type $h$ spectrum. Again, this localisation does not depend on choices by the results of Hopkins and Smith.

\begin{definition}\label{definition:convergence}
For $X$ a pointed space we say \emph{the $v_h$-periodic Goodwillie tower of $X$ converges} if the map
\begin{equation*}
\Phi_v X \rightarrow  \myholim{n}  \ \Phi_v P_n X
\end{equation*}
is an equivalence. 
\end{definition}

This definition is easily shown to be independent of $V$, although this will not concern us here. 
In Section $4$ of \cite{2}, Arone and Mahowald establish this convergence for spheres:

\begin{theorem}[Arone--Mahowald]
\label{thm:AM}
The $v_h$-periodic Goodwillie tower of $S^j$ converges for every $j \geq 1$. Moreover the tower is \emph{finite}, meaning it becomes constant at a finite stage.
\end{theorem}
In Section \ref{sec:3} and \ref{sec:5}, we will prove that both statements fail on two natural classes of spaces:
\begin{theorem}[\ref{thm:wedgesofspheres}]
The $v_h$-periodic Goodwillie tower is infinite and fails to converge on wedges of spheres (of dimension at least $2$). %added brackets for consistency with Moore space result
\end{theorem}
\begin{theorem}[\ref{thm:Moorespaces}]
The $v_1$-periodic Goodwillie tower of a Moore space $S^\ell/p$ is infinite and fails to converge (for $p$ odd and $\ell\geq 5$).
\end{theorem}
To prove this last theorem, %comma
we will analyse the Goodwillie layers on the cofibre of a given map $f: X \rightarrow Y$ of pointed spaces. Given a spectrum $Z$, we will write $\mathbf{D}_n(Z)=(\partial_n \id \wedge Z^{\wedge n})_{h\Sigma_n}$ for $\partial_n \id $ the Spanier-Whitehead dual of the  $n^{th}$ suspended partition complex. If $Z$ is the suspension spectrum of a space, then $\Omega^\infty \mathbf{D}_n(Z)$ gives the $n^{\mathrm{th}}$ layer of the Goodwillie tower of said space. We will often abuse notation and denote a space and its suspension spectrum by the same symbol. %look at this sentence [changed `name' to `symbol']
The crucial tool in our considerations is a filtration of the spectrum $\mathbf{D}_n(\cof(f))$, which is established in Lemma \ref{lem:filtrationGoodwillie}.

Finally, in the proof of Theorem \ref{lem:splittower} we make the following observation, which may be of some independent interest:

\begin{proposition}
Suppose $X$ is a finite type $h$ space with a $v_h$ self-map $\Sigma^d X \rightarrow X$. Then the tower $\{\Phi_v P_n(\Sigma^2X)\}_{n \geq 1}$ splits, meaning that there are equivalences
\begin{equation*}
\Phi_v P_n(\Sigma^2 X) \simeq \bigoplus_{k=1}^n \Phi_v D_k(\Sigma^2 X)
\end{equation*}
natural in $n$.
\end{proposition}
\begin{remark}
We have learnt recently that at height $h=1$, the divergence of the $v_1$-periodic Goodwillie tower on Moore spaces can also be deduced from Theorem 1.1. of \cite{ls}, which constructs so-called ``nonloopable $K/p_\ast$-equivalences''.
\end{remark}

\section*{Acknowledgements}
The authors would like to thank the Hausdorff Institute for Mathematics in Bonn where the final stages of this paper were completed. The first author was partially supported by an ERP scholarship of the German Academic Scholarship Foundation. We would also like to thank the anonymous referee for pointing out a mistake in an earlier version of Lemma \ref{lem:filtrations}.%thanks to referee

\section{Goodwillie Towers on Wedge Sums}
\label{sec:calculus}

In this section, we will compute the Goodwillie tower of the identity on wedges of spaces. We begin with the following lemma: 
\begin{lemma}\label{towersonwedges}\label{lem:key}
Let $G\colon \mathcal{S}_* \rightarrow \mathcal{S}_*$ be a functor preserving weak equivalences, filtered colimits, and the basepoint (i.e. $G$ is reduced). Given a positive integer $k$, we write $\bigvee\colon \mathcal{S}_*^{\times k} \longrightarrow \mathcal{S}_*$ for the iterated wedge sum functor and set $F= G \circ \bigvee \in \Fun(\mathcal{S}_*^{\times k} ,\mathcal{S}_*)$. 
Let $U_n^k$ be  the ordered subset of $[n]^{\times k}$ consisting of all tuples $(a_1,\dots,a_k)$ satisfying $a_1 + \dots + a_k \leq n$. Then there are canonical equivalences of functors $\mathcal{S}_*^{\times k} \rightarrow \mathcal{S}_*$:
\begin{equation*}
\bigl(P_n G\bigr) \circ \bigvee  \longrightarrow  P_n F \longrightarrow \myholim{\vv n \in U_n^k}\  P_{\vv n} F.
\end{equation*}
The $P_n$ on the left and the middle refers to single-variable calculus, the $P_{\vv n}$ on the right-hand side to multivariable calculus.
\end{lemma}
\begin{remark}
In this paper we focus on calculus for the category of pointed spaces. Goodwillie calculus can be developed in more general settings, e.g. for model categories satisfying appropriate conditions \cite{kuhngoodwillie, pereira} or certain kinds of $\infty$-categories (Chapter 6 of \cite{A}), and the appropriate analogue of Lemma \ref{lem:key} holds there as well.
\end{remark}
\begin{proof}[Proof of \ref{towersonwedges}]
The left arrow is an equivalence since the functor $\bigvee\colon \mathcal{S}_\ast^{\times k} \rightarrow \mathcal{S}_\ast$ preserves homotopy colimits. The existence of the right morphism follows from the observation that functors 
$\mathcal{S}^{\times k}_\ast \rightarrow \mathcal{S}_\ast$ which are 
 $\vv n$-excisive as multivariable functors are $\sum_i n_i$-excisive when considered as functors in one variable \mbox{(cf. Lemma 6.6 of \cite{1}).}
  
We now prove by induction that the composition of the two arrows is an equivalence, which will conclude the proof of the statement. Assume that the map is an equivalence for $n-1$. We have a diagram of fibre sequences
\[
\begin{tikzcd}
D_n G \circ \bigvee \ar{r}\ar{d} & \prod_{n_1+ \dots + n_k = n} D_{(n_1,\dots,n_k) }\left(G \circ \bigvee \right) \ar{d} \\
P_n G  \circ \bigvee \ar{r}\ar{d} & \myholim{\vv n \in U_k^n}  \ P_{\vv n} \left( G \circ \bigvee \right) \ar{d} \\
P_{n-1} G \circ \bigvee \ar{r} & \myholim{\vv{n-1} \in U_k^{n-1}}  \ P_{\vv{n-1}}\left( G \circ \bigvee \right). 
\end{tikzcd}
\]
% original in Lukas' TeX:
%\begin{diagram}[small]
%D_n G \circ \bigvee  & \rTo & \prod_{n_1+ \dots + n_k = n} D_{(n_1,\dots,n_k) }\left(G \circ \bigvee \right) \\
%\dTo & & \dTo \\
%P_n G  \circ \bigvee & \rTo & \myholim{\vv n \in U_k^n}  \ P_{\vv n} \left( G \circ \bigvee \right) \\
%\dTo & & \dTo \\
%P_{n-1} G \circ \bigvee & \rTo & \myholim{\vv{n-1} \in U_k^{n-1}}  \ P_{\vv{n-1}}\left( G \circ \bigvee \right). 
%\end{diagram}
A natural transformation between reduced $n$-excisive functors is an equivalence if and only if it induces an equivalence on each homogeneous layer $D_k$, for $k \leq n$. By our inductive hypothesis, it therefore suffices to prove that the top horizontal arrow in the diagram above is an equivalence. Observe that there are natural equivalences
\begin{eqnarray*}
D_n G(X_1 \vee \cdots \vee X_k) & \simeq & \Omega^\infty\bigl((\partial_n G \wedge (X_1 \vee \cdots \vee X_k)^{\wedge n})_{h\Sigma_n}\bigr) \\
& \simeq & \prod_{n_1 + \dots + n_k = n} \Omega^\infty\bigl((\partial_n G \wedge X^{n_1} \wedge \cdots \wedge X^{n_k})_{h(\Sigma_{n_1} \times \cdots \times \Sigma_{n_k})}\bigr).
\end{eqnarray*}
The latter expression is precisely the evaluation of the functor
\begin{equation*}
\prod_{n_1+ \dots + n_k = n} D_{(n_1,\dots,n_k) }\left(G \circ \bigvee \right) 
\end{equation*}
at $(X_1, \ldots, X_k)$ (compare Lemma 1.3 of \cite{3}) and it is straightforward to see that the horizontal map is compatible with these identifications. Hence the claim holds true for $n$.
\end{proof}

Every word $w\in \mathbf{L}_k$ gives rise to a functor
$ w: \mathcal{S}_\ast^{\times k} \rightarrow \mathcal{S}_\ast$
by smashing the factors in the order they appear in $w$.
The \textit{iterated Samelson product} yields a transformation
\[ w(\iota_{X_1}, \ldots, \iota_{X_k}): w(X_1,\dots , X_k) \rightarrow \Omega \Sigma( X_1\vee \dots\vee X_k). \]
Extending multiplicatively, we obtain a transformation 
\[\varphi_w:  \Omega \Sigma w(X_1,\dots , X_k) \rightarrow \Omega \Sigma( X_1\vee \dots \vee X_k). \]
Finally, we obtain a transformation
\[\sideset{}{'}\prod_{w\in \mathbf{L}_k} \Omega \Sigma  w(X_1,\dots , X_k)  \rightarrow \Omega \Sigma( X_1\vee \dots \vee X_k) \]
by multiplying all these maps in the order determined by $\mathbf{L}_k$ (compare Theorem 4.3.3 of \cite{4}).

\begin{theorem}[Hilton \cite{hilton}, Milnor\cite{milnor}]
The natural transformation 
\[ \sideset{}{'}\prod_{w\in \mathbf{L}_k} \Omega \Sigma  w(X_1,\dots , X_k) \rightarrow \Omega \Sigma( X_1\vee \dots\vee X_k) \]
is an equivalence.
\end{theorem}

Lemma \ref{lem:key} allows us to understand $P_n(\Omega\Sigma)(X_1 \vee \dots \vee X_k)$ in terms of multivariable calculus. The following lemma (which is also contained, with a different proof using a connectivity argument, in the proof of Lemma 1.4 of \cite{3}) will similarly help us understand the polynomial approximations to the domain of the Hilton-Milnor map:

\begin{lemma}[Arone--Kankaanrinta]
\label{lem:smashpowers}
Let $G\colon \mathcal{S}_* \rightarrow \mathcal{S}_*$ be a reduced functor preserving weak equivalences and filtered colimits. Consider a $k$-tuple of natural numbers $(a_1, \ldots, a_k)$ and define a functor $F\colon \mathcal{S}_*^{\times k} \rightarrow \mathcal{S}_*$ by
\begin{equation*}
F(X_1, \ldots, X_k) := G(X_1^{\wedge a_1} \wedge \cdots \wedge X_k^{\wedge a_k}).
\end{equation*}
Then for any $k$-tuple $\vv n = (n_1, \ldots, n_k)$ there is a natural equivalence
\begin{equation*}
P_{\vv n} F(X_1,\ldots, X_k) \xrightarrow{\sim} P_l G(X_1^{\wedge a_1} \wedge \cdots \wedge X_k^{\wedge a_k}),
\end{equation*}
where $l = \mathrm{min}_i \lfloor \frac{n_i}{a_i}\rfloor$.
\end{lemma}
\begin{proof}
To determine $P_{\vv n} F$ we may assume without loss of generally that $G$ is $N$-excisive for some $N$. In fact $N = n_1 + \ldots + n_k$ suffices. 
With this assumption, $G$ fits into a finite tower of fibrations
\[
\begin{tikzcd}
D_N G \ar{d} & D_{N-1} G \ar{d} & \cdots & D_2 G \ar{d} & \\
G \ar{r} & P_{N-1} G \ar{r} & \cdots \ar{r} & P_2 G \ar{r} & P_1 G.
\end{tikzcd}
\]
% original in Lukas' TeX:
%\begin{diagram}[small]
%D_N G  &  & D_{N-1} G & & \cdots & & D_2 G && \\
%\dTo & & \dTo && && \dTo &&\\
%G & \rTo & P_{N-1} G & \rTo & \cdots & \rTo & P_2 G & \rTo & P_1 G. \\
%\end{diagram}
Therefore $F(X_1, \ldots, X_k)$ arises from a finite tower of fibrations in which the successive fibres are
\begin{equation*}
\Omega^\infty\bigl((\partial_j G \wedge X^{\wedge ja_1} \wedge \cdots \wedge X^{\wedge ja_k})_{h\Sigma_j} \bigr)
\end{equation*}
for $0 \leq j \leq N$. Since the process of forming multiderivatives preserves fibre sequences of functors, it is clear that the only values of $\vv n$ for which the multiderivative $D_{\vv n} F$ can be nonzero are the multiples of $\vv a$, i.e., the tuples $(ja_1, \ldots, ja_k)$. For a general tuple $\vv n$, it follows that
\begin{equation*}
P_{\vv n} F(X_1,\ldots, X_k) \xrightarrow{\sim} P_l G(X_1^{\wedge a_1} \wedge \cdots \wedge X_k^{\wedge a_k})
\end{equation*}
where $l$ is the largest integer such that $la_i \leq n_i$ for every $1 \leq i \leq k$.
\end{proof}

We can combine these results to compute the Goodwillie tower on a wedge sum:
\begin{theorem}
\label{thm:HMGoodwillie}
Fix pointed connected spaces $X_1, \ldots, X_k$.
The Hilton--Milnor map refines to an equivalence of towers:
\[
\sideset{}{'}\prod_{w\in \mathbf{L}_k}\Omega P_{\lfloor \frac{n}{|w|}\rfloor} \Bigl(\Sigma w(X_1, \ldots, X_k) \Bigr) \xrightarrow{\sim} \Omega P_n (\Sigma X_1 \vee \cdots \vee \Sigma X_k).
\]
\end{theorem}

\begin{proof}
We use without further notice that we have $P_n(\Omega G \Sigma) \simeq \Omega \circ P_nG \circ \Sigma$ for any functor \mbox{$G\colon \mathcal{S}_* \rightarrow \mathcal{S}_*$}. Also, to avoid notational confusion, we will write $P_n \mathrm{id}(X)$ rather than the abbreviated $P_n X$ for the length of this proof. Forming $\vv{n}$-excisive approximations to functors of $k$ variables commutes with finite limits and filtered colimits, so that the Hilton--Milnor theorem implies an equivalence
\begin{equation*}
\sideset{}{'}\prod_{w\in \mathbf{L}_k} \Omega P_{\vv n}(\Sigma  w)(X_1,\ldots , X_k) \rightarrow \Omega P_{\vv n}(\Sigma(-\vee \dots\vee -))(X_1, \ldots, X_k).
\end{equation*}
We wish to form the homotopy limit over $\vv n \in U_k^n$ on both sides, where as before $U_n^k$ is the ordered subset of $[n]^{\times k}$ consisting of all tuples $(a_1,\dots,a_k)$ satisfying $a_1 + \dots + a_k \leq n $.

On the right-hand side, this will by Lemma \ref{lem:key} yield
\begin{equation*}
\Omega P_n \mathrm{id}(\Sigma X_1 \vee \dots \vee \Sigma X_k).
\end{equation*}
On the left-hand side, Lemma \ref{lem:smashpowers} implies that the homotopy limit 
\begin{equation*}
\myholim{\vv n \in U_k^n} P_{\vv n}(\Sigma w)(X_1, \ldots, X_k) 
\end{equation*}
may be identified with 
\begin{equation*}
P_l\mathrm{id}(\Sigma w(X_1, \ldots, X_k)),
\end{equation*}
for $l$ the largest integer satisfying $l(w_1 + \cdots + w_k) \leq n$. In other words $l = \lfloor \frac{n}{|w|} \rfloor$, which completes the proof.
\end{proof}

\section{Divergence on Wedges}\label{sec:3}
\label{sec:counterexamples}

Our description of Goodwillie towers on wedges in Theorem \ref{thm:HMGoodwillie} has the following straightforward consequence:

\begin{theorem}
\label{thm:HMproducts}
Consider pointed connected spaces $X_1, \ldots, X_k$. Then there is a natural equivalence
\begin{equation*}
\prod_{w \in \mathbf{L}_k} \myholim{n} \Phi_v P_n(\Sigma w(X_1, \ldots, X_k)) \rightarrow \myholim{n} \Phi_v P_n(\Sigma X_1 \vee \cdots \vee \Sigma X_k).
\end{equation*}
The product on the left is to be interpreted as the homotopy product.
\end{theorem}
\begin{proof}
The functor $\Phi_v$ commutes with finite homotopy limits and filtered colimits. Therefore Theorem \ref{thm:HMGoodwillie} gives an equivalence
\begin{equation*}
\sideset{}{'}\prod_{w\in \mathbf{L}_k} \Phi_v P_{\lfloor \frac{n}{|w|}\rfloor}(\Sigma  w(X_1,\ldots , X_k)) \rightarrow \Phi_v P_n(\Sigma X_1 \vee \cdots \vee \Sigma X_k).
\end{equation*}
Since only finitely many words $w \in \mathbf{L}_k$ have length not exceeding $n$, the product on the left has only finitely many nonzero factors and there is no need to distinguish between the weak product and the actual product. 
Therefore we find an equivalence
\begin{equation*}
\myholim{n}\prod_{w \in \mathbf{L}_k} \Phi_v P_{\lfloor \frac{n}{|w|}\rfloor}(\Sigma w(X_1, \ldots, X_k)) \rightarrow \myholim{n} \Phi P_n(\Sigma X_1 \vee \cdots \vee \Sigma X_k).
\end{equation*}
The homotopy limit can now be commuted past the product to obtain the result.
\end{proof}

\begin{corollary}\label{products}
Consider pointed connected spaces $X_1, \ldots, X_k$. Then the map
\begin{equation*}
\Phi_v(\Sigma X_1 \vee \cdots \vee \Sigma X_k) \rightarrow \myholim{n} \Phi_v P_n(\Sigma X_1 \vee \cdots \vee \Sigma X_k)
\end{equation*}
may be identified (up to equivalence) with the evident map
\begin{equation*}
\sideset{}{'}\prod_{w\in \mathbf{L}_k} \Phi_v(\Sigma  w(X_1,\ldots , X_k)) \rightarrow \prod_{w \in \mathbf{L}_k} \myholim{n} \Phi_v P_n(\Sigma w(X_1, \ldots, X_k)).
\end{equation*}
\end{corollary}

The previous corollary makes it easy to construct examples of spaces for which the $v_h$-periodic Goodwillie tower does not converge:

\begin{corollary} 
\label{cor:counterexample}
Consider pointed connected spaces $X_1, \ldots, X_k$ and assume that $\Phi_v(\Sigma w(X_1, \ldots, X_k))$ is not contractible for infinitely many $w \in \mathbf{L}_k$.  Then the canonical map
\[  \Phi_v(\Sigma X_1 \vee \cdots \vee \Sigma X_k) \rightarrow \myholim{n} \Phi_v P_n(\Sigma X_1 \vee \cdots \vee \Sigma X_k) \]
is \emph{not} an equivalence.
\end{corollary}
\begin{proof}
By Corollary \ref{products}, the map in question appears as the diagonal in the following diagram:
\[
\begin{tikzcd}
\displaystyle \sideset{}{'}\prod_{w\in \mathbf{L}_k} \Phi_v(\Sigma  w(X_1,\ldots , X_k)) \ar{r}\ar{d} &\displaystyle \sideset{}{}\prod_{w \in \mathbf{L}_k} \Phi_v(\Sigma w(X_1, \ldots, X_k)) \ar{d} \\
\displaystyle \sideset{}{'} \prod_{w \in \mathbf{L}_k} \displaystyle\myholim{n} \Phi_v P_n(\Sigma w(X_1, \ldots, X_k)) \ar{r} &\displaystyle \sideset{}{}\prod_{w \in \mathbf{L}_k} \displaystyle \myholim{n} \Phi_v P_n(\Sigma w(X_1, \ldots, X_k)).
\end{tikzcd}
\]
% \begin{diagram}\displaystyle
%\sideset{}{'}\prod_{w\in \mathbf{L}_k} \Phi_v(\Sigma  w(X_1,\ldots , X_k))& \rTo &\displaystyle \sideset{}{}\prod_{w \in \mathbf{L}_k} \Phi_v(\Sigma w(X_1, \ldots, X_k))  \\
%\dTo & & \dTo \\
%\displaystyle \sideset{}{'} \prod_{w \in \mathbf{L}_k} \displaystyle\myholim{n} \Phi_v P_n(\Sigma w(X_1, \ldots, X_k)) & \rTo &\displaystyle \sideset{}{}\prod_{w \in \mathbf{L}_k} \displaystyle \myholim{n} \Phi_v P_n(\Sigma w(X_1, \ldots, X_k)).
%\end{diagram}
The homotopy groups of an infinite (possibly restricted) product can simply be computed as the infinite product of homotopy groups (there is no $\mathrm{lim}^1$ term). For a given $\ell$, we obtain a square
\[
\begin{tikzcd}
\sideset{}{'}\prod_{w\in \mathbf{L}_k} \pi_\ell \Phi_v(\Sigma  w(X_1,\ldots , X_k)) \ar[>->]{r} \ar{d} &\displaystyle \sideset{}{}\prod_{w \in \mathbf{L}_k}\pi_\ell  \Phi_v(\Sigma w(X_1, \ldots, X_k)) \ar{d}  \\
\displaystyle \sideset{}{'} \prod_{w \in \mathbf{L}_k} \displaystyle\pi_\ell\,\myholim{n} \Phi_v P_n(\Sigma w(X_1, \ldots, X_k)) \ar[>->]{r} &\displaystyle \sideset{}{}\prod_{w \in \mathbf{L}_k} \displaystyle\pi_\ell \,\myholim{n} \Phi_v P_n(\Sigma w(X_1, \ldots, X_k)).
\end{tikzcd}
\]
% \begin{diagram}\displaystyle
%\sideset{}{'}\prod_{w\in \mathbf{L}_k} \pi_\ell \Phi_v(\Sigma  w(X_1,\ldots , X_k))& \rInto &\displaystyle \sideset{}{}\prod_{w \in \mathbf{L}_k}\pi_\ell  \Phi_v(\Sigma w(X_1, \ldots, X_k))  \\
%\dTo & & \dTo \\
%\displaystyle \sideset{}{'} \prod_{w \in \mathbf{L}_k} \displaystyle\pi_\ell\,\myholim{n} \Phi_v P_n(\Sigma w(X_1, \ldots, X_k)) & \rInto &\displaystyle \sideset{}{}\prod_{w \in \mathbf{L}_k} \displaystyle\pi_\ell \,\myholim{n} \Phi_v P_n(\Sigma w(X_1, \ldots, X_k)).
%\end{diagram}
If there exists a word $w$ such that the map 
\[
\pi_\ell \,\Phi_v(\Sigma  w(X_1,\ldots , X_k)) \ \ \longrightarrow\ \ \pi_\ell\,  \myholim{n} \Phi_v P_n(\Sigma w(X_1, \ldots, X_k)) \ \]
is \textit{not} a bijection, then the left vertical map of the last square also fails to be bijective. Since postcomposing a non-bijective map with an injection yields a non-bijective map, the diagonal fails to be bijective. We may therefore assume without restriction that both vertical legs are isomorphisms and it suffices to check that the top horizontal map is not surjective.

Combining the periodicity of the homotopy groups of $\Phi_v$ with the pigeonhole principle, we conclude that there is an integer $\ell$ with $\pi_\ell\Phi_v(\Sigma w(X_1, \ldots, X_k)) \neq 0$ for infinitely many $w \in \mathbf{L}_k$. This clearly implies that the top horizontal map is not a bijection and thereby establishes the claim.
\end{proof}

We can now deduce:

\begin{theorem}
\label{thm:wedgesofspheres}
The $v_h$-periodic Goodwillie tower is infinite and fails to converge on wedges of spheres (of dimension at least $2$).
\end{theorem}
\begin{proof}
The divergence follows by applying Corollary \ref{cor:counterexample} in the case where each $X_i$ is a sphere of dimension at least 1. Indeed, the assumption of the corollary holds true since each $\Sigma w(X_1, \ldots, X_k)$ is another sphere of dimension at least 2. Such a sphere has nonvanishing $v_h$-periodic homotopy groups. 

To see that the $v_h$-periodic Goodwillie tower of $\Sigma X_1 \vee \cdots \vee \Sigma X_k$ does not become constant at any finite stage, we first recall that the $v_h$-periodic Goodwillie tower of a sphere only becomes constant at stage $p^h$ or $2p^h$, depending on the parity of the dimension of the sphere. It then follows from Theorem \ref{thm:HMGoodwillie} that in the tower for $\Sigma X_1 \vee \cdots \vee \Sigma X_k$, the contribution from the factor corresponding to a word $w$ only becomes constant at stage $|w|p^h$ or $2|w|p^h$. Clearly there is no bound on these numbers as $w$ runs through $\mathbf{L}_k$.
\end{proof}

\section{Goodwillie derivatives on cofibres}

In this section we analyse the layers in the Goodwillie tower of the identity functor when evaluated on a cofibre $\cof(f)$ in terms of the map $f\colon X\rightarrow Y$. In fact, our analysis concerns the free $\mathcal{O}$-algebra on $\cof(f)$ for $\mathcal{O}$ any operad in spectra. 
The Goodwillie layers of a space arise when we take $\mathcal{O}$ to be the ``shifted Lie operad" (cf. \cite{ching}) -- we will explain this in more detail later.
In a later section, we will  apply the methods developed in this section to the example of a Moore space, i.e. the cofibre of the `multiplication by $p$' map $S^\ell \xrightarrow{p} S^\ell$.

For now, let $f\colon E \rightarrow F$ be a map of spectra. We first observe that we can write $\cof(f)^{\wedge n}$ as the total cofibre of a cubical $\Sigma_n$-diagram. We refer to \cite{cubical} for a detailed exposition of cubical homotopy theory. Indeed, write $\mathcal{P}(\underline{n})$ for the power set of $\underline{n} = \{1, \ldots, n\}$, which is partially ordered under inclusion. Let $\Delta^1$ denote the poset $\{0 < 1\}$ and observe that there is an evident identification $\mathcal{P}(\underline{n}) \cong (\Delta^1)^{\times n}$. Write $\mathrm{Sp}$ for the category of spectra and define a diagram $C_n(f):\mathcal{P}(\underline{n}) \rightarrow \mathrm{Sp}$ by taking the composition
\[
\mathcal{P}(\underline{n}) \cong (\Delta^1)^{\times n} \xrightarrow{f^{\times n}} \mathrm{Sp}^{\times n} \xrightarrow{\wedge^n}  \mathrm{Sp}.
\]
This is a cubical diagram with initial vertex $E^{\wedge n}$ and final vertex $F^{\wedge n}$. For a general subset $S \subset \underline{n}$, the corresponding value of $C_n(f)$ is given by  the smash product of $|S|$ copies of $Y$ with $|\underline{n} \backslash S|$ copies of $X$. The smash product $\cof(f)^{\wedge n}$ is now the total cofibre of $C_n(f)$. Writing $\mathcal{P}^-(\underline{n})  = \mathcal{P}(\underline{n}) \backslash \{ \underline{n}\}$, we have
\begin{equation*}
\cof(f)^{\wedge n} \simeq \cof (\myhocolim{} \, C_n(f)|_{\mathcal{P}^-(\underline{n})}  \rightarrow  F^{\wedge n}).
\end{equation*}
 
The punctured cube $\mathcal{P}^-(\underline{n})$ comes with an evident filtration by the size of subsets:
\[
\mathcal{P}^-(\underline{n})_{\leq 0} \subset \mathcal{P}^-(\underline{n})_{\leq 1} \subset \dots \subset \mathcal{P}^-(\underline{n})_{\leq {n-1}} = \mathcal{P}^-(\underline{n}).
\]
We shall write 
\[
C_n^k(f) := C_n(f)|_{\mathcal{P}^-(\underline{n})_{\leq k}}.
\]
We obtain a corresponding filtration of $\cof(f)^{\wedge n}$ by the (na\"ive) $\Sigma_n$-spectra $\cof^k_n(f)$ \mbox{defined by}
\begin{equation*}
\cof^k_n(f) := \cof\left(\myhocolim{}\,C_n^k(f) \rightarrow F^{\wedge n}\right).
\end{equation*}
for $k=0,\ldots,n-1$. We decree that $\cof^{-1}_n(f)  = 0$. 
\begin{lemma}
\label{lem:filtrations}
Given a map $f: E\rightarrow F$ of spectra, the filtration of $\Sigma_n$-spectra
\[ \cof^0_n(f) \rightarrow \dots \rightarrow \cof^{n-1}_n(f)  \cong \cof(f)^{\wedge n} \]
defined above has associated graded spectra $\gr_k(\cof(f)^{\wedge n})= \cof\left(\cof^{k-1}_n(f) \rightarrow \cof^k_n(f) \right) $ given by
\[
  \gr_k(\cof(f)^{\wedge n})  =
   \begin{cases}
     \cof(E^{\wedge n} \rightarrow F^{\wedge n}) & \text{for }   k=0, \\
     \Sigma \Ind_{\Sigma_{n-k}\times \Sigma_{k}}^{\Sigma_n} \left( E^{\wedge (n-k)}\wedge  \cof(f)^{\wedge k}\right) & \text{for }  1 \leq k \leq n-1.
   \end{cases}
\]
\end{lemma}

\begin{remark}
There is another (perhaps more standard) filtration of $\cof(f)^{\wedge n}$ defined as follows. One considers the map $F \rightarrow \cof(f)$ and smashes it with itself $n$ times to again obtain a cubical diagram $
\mathcal{P}(\underline{n}) \rightarrow \mathrm{Sp}.$ A $\Sigma_n$-equivariant filtration of $\cof(f)^{\wedge n}$ can be defined by taking the homotopy colimit over the restriction of this cubical diagram to each of the $\mathcal{P}(\underline{n})_{\leq k}$. The associated graded spectra of this filtration are
\begin{equation*}
\Ind_{\Sigma_{n-k} \times \Sigma_k}^{\Sigma_n} ((\Sigma E)^{\wedge (n-k)} \wedge F^{\wedge k}).
\end{equation*}
This filtration will not suffice for our purposes because the associated graded spectra do not depend on the map $f$.
\end{remark}

\begin{proof}[Proof of Lemma \ref{lem:filtrations}]
The claim for $k=0$ is evident. Assume $1\leq k \leq n-1$.
We can compute $\gr_k(\cof(f)^{\wedge n})$ as the total cofibre of the square 
\[
\begin{tikzcd}
\myhocolim{} \, C_n^{k-1}(f) \ar{r}\ar{d} & F^{\wedge n}  \ar{d} \\ 
\myhocolim{} \, C_n^k(f) \ar{r} & F^{\wedge n}, 
\end{tikzcd}
\]
%\begin{diagram}[small]
%\myhocolim{} \, C_n^{k-1}(f) & \rTo & F^{\wedge n}  \\ 
%\dTo & & \dTo &   \\
%\myhocolim{} \, C_n^k(f) & \rTo & F^{\wedge n}, 
%\end{diagram}
which shows that 
\[ \gr_k(\cof(f)^{\wedge n}) \cong \Sigma  \cof\left( \myhocolim{} \, C_n^{k-1}(f) \rightarrow \myhocolim{} \,C_n^{k}(f)  \right). \]
The cofibre on the right-hand side can alternatively be computed by first taking the homotopy left Kan extension of the diagram $C_n^{k-1}(f)$ from $\mathcal{P}^-(\underline{n})_{\leq k-1}$ to $\mathcal{P}^-(\underline{n})_{\leq k}$ (we denote this extension by $LC_n^{k-1}(f)$), then taking the cofibre of the natural transformation from $LC_n^{k-1}(f)$ to the diagram $C_n^k(f)$, and finally taking the homotopy colimit over $\mathcal{P}^-(\underline{n})_{\leq k}$. Note that the diagram $LC_n^{k-1}(f)$ agrees with $C_n^k(f)$ when evaluated on subsets of size smaller than $k$. If $S$ has size exactly $k$, then one easily computes
\begin{equation*}
LC_n^{k-1}(f)(S) \simeq E^{\wedge (\underline{n} - S)} \wedge \myhocolim{} \, C_S(f)|_{\mathcal{P}^-(S)}, 
\end{equation*}
where $C_S(f)$ is equivalent to $C_k(f)$ after identifying $S$ with $\underline{k}$. However, it is of course better to think of it as a $k$-dimensional cubical diagram indexed on subsets of $S$. It follows that
\begin{eqnarray*}
& & \cof\bigl(\myhocolim{}\  LC_n^{k-1}(f) \rightarrow \myhocolim{} \ C_n^k(f)\bigr) \\& \simeq & \bigoplus_{|S| = k} E^{\wedge (\underline{n} - S)} \wedge \cof(\myhocolim{} \, C_S(f)|_{\mathcal{P}^-(S)} \rightarrow F^{\wedge S}) \\
& \simeq & \bigoplus_{|S| = k} E^{\wedge (\underline{n} - S)} \wedge \cof(f)^{\wedge S} \\
& \simeq & \mathrm{Ind}_{\Sigma_{n-k} \times \Sigma_k}^{\Sigma_n}\bigl( E^{\wedge (n-k)} \wedge \cof(f)^{\wedge k} \bigr),
\end{eqnarray*}
where we regard $\Sigma_{n-k} \times \Sigma_k$ as a subgroup of $\Sigma_n$ using the standard inclusion. This group acts on $E^{\wedge (n-k)} \wedge \cof(f)^{\wedge k}$ in the evident manner. The asserted claim follows.\end{proof}

Now let $\mathcal{O}$ be an operad in spectra, i.e., a sequence of spectra $\{\mathcal{O}(n)\}_{n \geq 0}$ with symmetric group actions and composition maps satisfying the usual axioms. The free (homotopy) $\mathcal{O}$-algebra on a spectrum $X$ is, up to equivalence, described by the formula
\begin{equation*}
\bigoplus_{n \geq 0} (\mathcal{O}(n) \wedge X^{\wedge n})_{h\Sigma_n}.
\end{equation*}
We apply Lemma \ref{lem:filtrations} to free $\mathcal{O}$-algebras on cofibres in two special cases of interest:

\begin{corollary}
Let $f\colon E \rightarrow F$ be a map of spectra. Then the $n^{th}$ summand in the free $\mathbf{E}_\infty$-algebra on $\mathrm{cof}(f)$ has a finite filtration whose associated graded spectra are given by
\[
  \gr_k\bigl( \mathrm{cof}(f)^{\wedge n}_{h\Sigma_n}\bigr)  =
   \begin{cases}
     \cof(E^{\wedge n} \rightarrow F^{\wedge n})_{h\Sigma_n} & \text{for }   k=0, \\
     \Sigma E^{\wedge (n-k)}_{h\Sigma_{n-k}} \wedge \cof(f)^{\wedge k}_{h\Sigma_k} & \text{for }  1 \leq k \leq n-1.
   \end{cases}
\]
\end{corollary}

The example we are most interested in here is the case where $\mathcal{O} = \partial_* \mathrm{id}$ is the operad given by the derivatives of the identity on $\mathcal{S}_*$. Recall that $\partial_n \mathrm{id}$ is given by the Spanier-Whitehead dual of the suspended \emph{partition complex} \cite{2}, \cite{Joh}.  More precisely, we write $\Pi_n$ for the poset of proper nondiscrete partitions of the set $\{1,\ldots,n\}$ and $|\Pi_n|$ for the geometric realisation of its nerve. We then have a $\Sigma_n$-equivariant equivalence $\partial_n \mathrm{id} \simeq \DD(\Sigma |\Pi_n|^{\diamond})$. Here $(-)^\diamond$ denotes the unreduced suspension and $\DD$ stands for Spanier-Whitehead duality.

For a spectrum $E$, we write
\begin{equation*}
\mathbf{D}_n E = (\partial_n \mathrm{id} \wedge E^{\wedge n})_{h\Sigma_n}.
\end{equation*}
If $E = \Sigma^\infty X$ for some pointed space $X$, then  this agrees with the $n$th Goodwillie layer of $X$. We remark once more that we often abuse notation and denote a space and its suspension spectrum by the same name.

\begin{lemma}
\label{lem:filtrationGoodwillie}
Let $f\colon E \rightarrow F$ be a map of spectra. Then $\mathbf{D}_n\cof(f)$ has a finite filtration whose $k^{th}$ graded piece $ \gr_k\bigl(\mathbf{D}_n\cof(f)\bigr)$ is given by   
\[
   \begin{cases}
  \ \ \ \    \cof(\mathbf{D}_n E \rightarrow \mathbf{D}_n F) & \text{for }   k=0, \\ \\
   \displaystyle \bigoplus_{\substack{d \ |\ k,n-k \\ w\in B(\frac{n-k}{d}, \frac{k}{d})}} \Sigma  \mathbf{D}_d\Sigma((\Sigma^{-1}E)^{\wedge\frac{n-k}{d}} \wedge (\Sigma^{-1}\cof(f))^{\wedge\frac{k}{d}}) & \text{for }  1 \leq k \leq n-1.
   \end{cases}
\]
Here the sum ranges over numbers $d$ dividing both $k$ and $n-k$ and elements of the set $B(\frac{n-k}{d},\frac{k}{d})$. Recall that $B(i,j)$ denotes the set of words in the basis $\mathbf{L}_2$ which involve the first letter $i$ times and the second letter $j$ times.
\end{lemma}

\begin{proof}
The case $k=0$ follows directly from Lemma \ref{lem:filtrations} after smashing with $\partial_n\mathrm{id}$ and taking homotopy orbits for $\Sigma_n$. For  $k \geq 1$, we combine Lemma \ref{lem:filtrations} with the projection formula and read off an equivalence between $\bigl(\partial_n \mathrm{id} \wedge \gr_k(\cof(f)^{\wedge n})\bigr)_{h\Sigma_n}$ and 
\begin{equation*}
 \Sigma\left((\Res^{\Sigma_n}_{\Sigma_{n-k} \times \Sigma_k}\partial_n\mathrm{id}) \wedge E^{\wedge (n-k)}\wedge\cof(f)^{\wedge k}\right)_{h(\Sigma_{n-k} \times \Sigma_k)}.
\end{equation*}
Theorem $0.1$ of \cite{3} gives, for $k_1 + \cdots + k_m = n$, a natural equivalence between
\begin{equation*}
\bigl((\Res^{\Sigma_n}_{\Sigma_{k_1} \times \cdots \times \Sigma_{k_m}}\partial_n\mathrm{id}) \wedge Y_1^{\wedge k_1} \wedge \cdots \wedge Y_m^{\wedge k_m}\bigr)_{h(\Sigma_{k_1} \times \cdots \times \Sigma_{k_m})} 
\end{equation*}
and
\begin{equation*}
\bigoplus_{\substack{d \ |\ k_1, \ldots, k_m\\ B(\frac{k_1}{d}, \ldots, \frac{k_m}{d})}}  \Bigl(\partial_d \mathrm{id} \wedge \left(\Sigma((\Sigma^{-1}Y_1)^{\wedge \frac{k_1}{d}} \wedge \cdots \wedge (\Sigma^{-1}Y_m)^{\wedge \frac{k_m}{d}})\right)^{\wedge d} \Bigr)_{h\Sigma_d} 
\end{equation*}
when each $Y_i$ is of the form $\Sigma^\infty\Sigma X_i$ for some connected space $X_i$. In other words, the $(k_1, \ldots,k_m)$-homogeneous functors from spaces to spectra defined by these two expressions are naturally equivalent when evaluated on suspensions of connected spaces. But using Goodwillie's correspondence between homogeneous functors from spaces to spectra and homogeneous functors from spectra to spectra (implemented by taking derivatives, cf. \cite{1}), it follows that the equivalence above can in fact be defined for any collection of spectra $Y_1, \ldots, Y_n$ (since the Goodwillie derivatives of a functor only depend on its values on highly connected spaces). %added sentence in last bracket
Applying this equivalence to our expression for $\bigl(\partial_n \mathrm{id} \wedge \gr_k(\cof(f)^{\wedge n})\bigr)_{h\Sigma_n}$ above gives the conclusion of the lemma.
\end{proof}

\begin{remark}
A different proof of Lemma \ref{lem:filtrationGoodwillie} can be given using recent work of Arone and the first author (cf. Theorem 5.10. of \cite{ab}), which studies the (unstable) equivariant homotopy type of the partition complex directly.
\end{remark}

\section{Divergence on Moore Spaces}\label{sec:5} 

We expect that the results of Cohen--Moore--Neisendorfer \cite{4} on splittings of the loop space of a Moore space together with methods analogous to our Theorem \ref{thm:HMGoodwillie} should give a detailed understanding of the $v_1$-periodic Goodwillie tower of a Moore space. However, our goal in this section is only to prove divergence, which we achieve by a series of cheap (but rather effective) swindles. To present our arguments in their simplest form we take $p$ to be an odd prime, but there is no essential difficulty in covering the case $p=2$ as well.
\begin{lemma}
\label{lem:Moorelayers}
Let $M^\ell = S^\ell/p$ be a mod $p$ Moore space. Then infinitely many Goodwillie layers of $M^\ell$ have nonzero $v_1$-periodic homotopy groups.
\end{lemma}
In order to prove this lemma we introduce a certain notion of Euler characteristic:
\begin{definition}
Let $N$ be a finitely generated graded module over the graded field $\mathbb{F}_p[u^{\pm 1}]$ with $|u|=2$. Then we define the Euler characteristic of $N$ by
\begin{equation*}
\chi(N) = \mathrm{rk}_{\mathbb{F}_p[u^{\pm 1}]} N^{\mathrm{ev}} - \mathrm{rk}_{\mathbb{F}_p[u^{\pm 1}]} N^{\mathrm{odd}}.
\end{equation*}
Here $N^{\mathrm{ev}}$ (resp. $N^{\mathrm{odd}}$) is the even-dimensional (resp. odd-dimensional) part of $N$ and $\mathrm{rk}_{{\mathbb{F}_p[u^{\pm 1}]}}$ denotes the rank of a module. 
\end{definition}

Now say $N$ is a finitely generated graded module as in the previous definition equipped with a differential $d$ of odd degree, i.e., $d = d_0 \oplus d_1$ with
\begin{equation*}
d_0\colon N^{\mathrm{ev}} \rightarrow N^{\mathrm{odd}} \quad\quad \text{and} \quad\quad d_1\colon N^{\mathrm{odd}} \rightarrow N^{\mathrm{ev}}.
\end{equation*}
Then clearly $\chi(H_*(N, d)) =\chi(N)$ since
\begin{eqnarray*} 
  & \bigl(\mathrm{rk}_{\mathbb{F}_p[u^{\pm 1}]}\mathrm{ker}(d_0) - \mathrm{rk}_{\mathbb{F}_p[u^{\pm 1}]}\mathrm{im}(d_1)\bigr) - \bigl(\mathrm{rk}_{\mathbb{F}_p[u^{\pm 1}]}\mathrm{ker}(d_1) - \mathrm{rk}_{\mathbb{F}_p[u^{\pm 1}]}\mathrm{im}(d_0)\bigr) \\
  = & \bigl(\mathrm{rk}_{\mathbb{F}_p[u^{\pm 1}]}\mathrm{ker}(d_0) + \mathrm{rk}_{\mathbb{F}_p[u^{\pm 1}]}\mathrm{im}(d_0)\bigr) - \bigl(\mathrm{rk}_{\mathbb{F}_p[u^{\pm 1}]}\mathrm{im}(d_1) +  \mathrm{rk}_{\mathbb{F}_p[u^{\pm 1}]}\mathrm{ker}(d_1)\bigr).
\end{eqnarray*}

% add some expository comments about the proof strategy here

\begin{proof}[Proof of  Lemma \ref{lem:Moorelayers}]
It suffices to show that infinitely many Goodwillie layers of $M^\ell$ have nonvanishing (completed) %added completed
 $p$-adic $K$-theory (e.g. compare 3.7 of \cite{bousfieldtelescopic}). We will achieve this by analyzing the $K$-homology of $\mathbf{D}_n M^\ell$ and showing that it is a graded $\mathbb{F}_p[u^{\pm 1}]$-module of non-zero Euler characteristic for infinitely many values of $n$.

Let $n$ be a \textit{prime} and consider the filtration of $\mathbf{D}_n M^\ell$ given by Lemma \ref{lem:filtrationGoodwillie}, i.e., we take $E=F=S^\ell$ and $f$ the multiplication by $p$. Since by \cite{2} the $v_1$-periodic Goodwillie tower of a sphere becomes constant after stage $p$ or $2p$ (depending on the parity of $\ell$), we know that for $n > 2p$ the $k=0$ graded piece of that filtration is null. We can use the $p$-adic $K$-theory of the remaining pieces with $1 \leq k \leq n-1$ as input for a spectral sequence, consisting of $n-1$ lines, converging to the $p$-adic $K$-theory of $\mathbf{D}_n M^\ell$. Since $n$ is prime, the greatest common divisor of $k$ and $(n-k)$ is $1$. By Lemma \ref{lem:filtrationGoodwillie}, the $k^{th}$ graded piece $ \gr_k (\mathbf{D}_n M^\ell)$ is given by the spectrum
\begin{equation*}
 \bigoplus_{\substack{ w\in B( n-k, k )}}     \Sigma^{2-n+\ell(n-k)} (M^\ell)^{\wedge k}.   
\end{equation*}
Neisendorfer (cf. Theorem 4.1. in \cite{5}) proves that
\begin{equation*}
M^{\ell_1} \wedge M^{\ell_2} \simeq M^{\ell_1 + \ell_2} \vee M^{\ell_1 + \ell_2 + 1}
\end{equation*}
whenever $\ell_1,\ell_2\geq 1$.
Iterating this splitting shows that
\begin{equation*}
(M^{\ell})^{\wedge k} \simeq \bigvee_{j=0}^k {{k-1}\choose{j}} M^{k\ell + j}.
\end{equation*} 
In particular, the associated graded spectra $\gr_k (\mathbf{D}_n M^\ell)$ are themselves wedge sums of Moore spectra and their $p$-adic $K$-theory is therefore $p$-torsion. Any finitely generated $p$-torsion abelian group is an $\mathbb{F}_p$-vector space in a unique way, and a homomorphism between such groups is automatically $\mathbb{F}_p$-linear. Therefore, the $E_1$-page of our spectral sequence is a module over the graded field $\mathbb{F}_p[u^{\pm 1}]$, with $u$ denoting the Bott class, and the differentials are $\mathbb{F}_p[u^{\pm 1}]$-linear. Moreover, all differentials are of odd degree, so that the Euler characteristic of the $E_1$-page is the same as that of the $E_\infty$-page. The Euler characteristic of $K_*(M^\ell)$ is of course $1$ or $-1$, depending on the parity of $\ell$. Note that if $k \geq 2$, then the Euler characteristic of $K_*((M^{\ell})^{\wedge k})$ is zero because the alternating sum of binomial coefficients vanishes. Therefore the corresponding lines in the spectral sequence, being sums of such modules, have zero Euler characteristic. For the $k=1$ line, however, the spectrum under consideration is
\begin{equation*}
    \bigoplus_{\substack{ w\in B( n-1, 1 )}}     \Sigma^{2-n+\ell(n-1)} M^\ell.
\end{equation*}
Its $p$-adic $K$-theory is either completely in even or completely in odd degrees and in particular has nonzero Euler characteristic. Thus the $E_\infty$-page has nonzero Euler characteristic as well. In particular it is nontrivial, so that $K_*(\mathbf{D}_n M^\ell)$ cannot vanish.
\end{proof}

Finally we wish to prove Theorem \ref{thm:Moorespaces}, which states that the $v_1$-periodic Goodwillie tower cannot converge on $M^\ell$. 
To do this, it is useful to know that the tower
$
\{\Phi_v P_n(M^\ell)\}_{n \geq 1}
$
is \emph{split} (at least for $\ell$ large enough), meaning that each stage is simply the finite sum of its homogeneous layers:
\begin{equation*}
\Phi_v P_n(M^\ell) \simeq \bigoplus_{k=1}^n \Phi_v \Omega^\infty(\mathbf{D}_k(M^\ell)).
\end{equation*}

\begin{theorem}
\label{lem:splittower}
If $p$ is odd and $\ell \geq 5$, the tower $\{\Phi_v P_n(M^\ell)\}_{n \geq 1}$ is split.
\end{theorem}
\begin{proof} 
The following proof is a condensed version of a more detailed analysis of Goodwillie calculus in $v_h$-periodic homotopy theory appearing in \cite{heuts}.
Following Bousfield \cite{bousfieldloc}\cite{bousfieldtelescopic}, we let $V_{h+1}$ be a finite space of type $h+1$ which is also a suspension.
Write $d_h$ for the dimension in which the first nonvanishing homotopy group of $V_{h+1}$ occurs.
We denote the category of  $d_h$-connected, pointed, and $p$-local spaces by $\mathcal{S}_{\ast}^{(p)}\langle d_h \rangle$. 
Consider the full subcategory 
$L_h^f\mathcal{S}_{\ast}^{(p)}\langle d_h \rangle$ spanned by all spaces $Y$ for which the map $ Y\rightarrow \Map(V_{h+1},Y) $
induced by pullback against $V_{h+1} \rightarrow \ast$ is an equivalence of unpointed spaces.
We obtain the left Bousfield localisation
\[
\begin{tikzcd}
\mathcal{S}_{\ast}^{(p)}\langle d_h \rangle \ar[shift left]{r}{L_h^f} & L_h^f\mathcal{S}_{\ast}^{(p)}\langle d_h \rangle. \ar[shift left]{l}{\iota}
\end{tikzcd}
\]

%\begin{diagram}
%\mathcal{S}_{\ast}^{(p)}\langle d_h \rangle & & \pile{\rTo^{L_h^f} \\ \lTo_\iota } & & L_h^f\mathcal{S}_{\ast}^{(p)}\langle d_h \rangle
%\end{diagram}

In \cite{bousfieldtelescopic}, Bousfield constructs a left Quillen functor 
\[\Theta\colon \mathrm{Sp}_{T(h)}\rightarrow L_h^f\mathcal{S}_{\ast}^{(p)}\langle d_h \rangle, \]
where $\mathrm{Sp}_{T(h)}$ denotes the localisation of the category of spectra with respect to $T(h)$, as before. This functor gives a derived left adjoint to the Bousfield--Kuhn functor $\Phi$, although we will not need this.

We state several facts about the left adjoint $L_h^f$:
\begin{itemize}
\item[(A)] The functor $L_h^f$ does not affect $v_h$-periodic homotopy groups.
\item[(B)] The functor $L_h^f$ preserves homotopy colimits.
\item[(C)] Any space $X$ with a $v_h$-self map is `almost' in the image of $\Theta$, in the sense that for such $X$ there is an equivalence $L_h^f \Sigma^2 X \cong \Theta L_{T(h)}\Sigma^{\infty+2} X$.
\item[(D)] The functor $L_h^f: \mathcal{S}_{\ast}^{(p)}\langle d_h \rangle \rightarrow \mathcal{S}_{\ast}^{(p)}\langle d_h \rangle $ preserves finite homotopy limits. Here and in the remainder of this article, we  omit the functor $\iota$ from our notation.
\end{itemize}
Statement (A) is proven in 4.6 of \cite{bousfieldtelescopic}. For (B), we use that $L_h^f$ is a left Quillen functor. Statement (C) is established in 5.9 of \cite{bousfieldtelescopic}. Statement (D) is Theorem 3.8 of \cite{heuts}. In fact, all four of these statements are summarised in Section 3 of op. cit.

We can relate the Goodwillie towers of the identity on the homotopy theories ${\mathcal{S}^{(p)}_\ast}\langle d_h \rangle$, and $L_h^f {\mathcal{S}_\ast^{(p)}}\langle d_h \rangle$ as follows. First, for $X$ a $d_h$-connected space, we have
\[ P_n(\id_{\mathcal{S}_\ast^{(p)}})(X) \cong P_n(\id_{\mathcal{S}_\ast^{(p)} \langle d_h \rangle})(X). \]
We now use facts (B) and  (D) to compute 
\[ L_h^f P_n(\id_{\mathcal{S}_\ast^{(p)}\langle d_h \rangle}) (X) \cong  P_n(\id_{L_h^f \mathcal{S}_\ast^{(p)}\langle d_h \rangle }) (L_h^f X).
\]
Here we used that precomposing with a functor preserving homotopy colimits and postcomposing with a functor preserving finite homotopy limits and filtered homotopy colimits commutes with the $n$-excisive approximation $P_n$.

Assuming that $X = \Sigma^2 X'$ for $X'$ a space with a $v_h$ self-map, we use (C) to obtain further equivalences
\[  P_n(\id_{L_h^f \mathcal{S}_\ast \langle d_h\rangle }) (L_h^f X) \cong  P_n(\id_{L_h^f \mathcal{S}_\ast \langle d_h\rangle}) (\Theta L_{T(h)} \Sigma^\infty X) \cong  P_n(\Theta ) (L_{T(h)} \Sigma^\infty X).
\]

Combining this with fact (A) and the straightforward observation that $\Phi_v$ preserves finite homotopy limits and filtered homotopy colimits, we finally get the chain of equivalences
\[\Phi_v P_n(X)  \cong \Phi_v P_n(\id_{\mathcal{S}_\ast^{(p)} \langle d_h \rangle})(X) \cong 
 \Phi_v L_h^f P_n(\id_{\mathcal{S}_\ast^{(p)} \langle d_h \rangle})(X)\cong P_n(\Phi_v \Theta ) (L_{T(h)} \Sigma^\infty X). \]

The composition of functors $\Phi_v\Theta$ is a functor from the category $\mathrm{Sp}_{T(h)}$ to itself and all such functors have split Goodwillie towers by results of Kuhn, see for example Theorem 1.1 of \cite{kuhntate}. This relies on the vanishing of $T(h)$-local Tate spectra.

We now apply these observations to the case where $h=1$ and where $X' = M^\ell$, which admits a $v_1$ self-map for $\ell \geq 3$. The integer $d_1$ can be taken to be $4$ in this case, by considering the type 2 complex $V_2$ which is the suspension of the cofiber of the $v_1$ self-map just mentioned. The analysis above shows that there is a weak equivalence of towers 
\begin{equation*}
\Phi_v P_n(M^{\ell +2}) \simeq P_n(\Phi_v \Theta)(L_{T(1)} \Sigma^\infty M^{\ell+2})
\end{equation*}
and the latter tower is split.
\end{proof}

We can finally prove our result on the $v_h$-periodic Goodwillie tower on a Moore space:
\begin{theorem} 
\label{thm:Moorespaces}
The $v_1$-periodic Goodwillie tower of a Moore space $S^\ell/p$ is infinite and fails to converge for $\ell \geq 5$ and $p$ an odd prime.
\end{theorem}

\begin{proof}
By the above discussion, the homotopy limit of the tower $\{\Phi_v P_n(M^\ell)\}_{n \geq 1}$ is the infinite product
\begin{equation*}
\prod_n \Phi_v \Omega^\infty \mathbf{D}_n(M^\ell).
\end{equation*}
Since infinitely many of the layers have nonvanishing homotopy groups (Lemma \ref{lem:Moorelayers}) and these homotopy groups are periodic, the pigeonhole principle applies again to guarantee the existence of an integer $j$ with
\begin{equation*}
\pi_j\bigl(\Phi_v \Omega^\infty \mathbf{D}_n(M^\ell)\bigr) \neq 0
\end{equation*}
for infinitely many values of $n$. By Cantor's diagonal slash, the group 
\begin{equation*}
\pi_j\bigl(\prod_n \Phi_v \Omega^\infty \mathbf{D}_n(M^\ell)\bigr)
\end{equation*}
must be uncountable. But this contradicts Thompon's calculation of the $v_1$-periodic homotopy groups of a Moore space in \cite{10}. Indeed, using the results of Cohen-Moore-Neisendorfer, he argues that the loop space of $M^\ell$ is a weak infinite product of certain spaces $S^m\{p\}$ and $T^m\{p\}$ and computes the $v_1$-periodic homotopy groups of these spaces. The result is his Theorem $1.1$, which in particular shows that these groups are countable.
\end{proof}

\begin{remark}
In Theorem 2.11 of \cite{heuts} it is in fact shown that $\Phi_v(M^\ell)$ is simply the direct \emph{sum} of all its homogeneous layers (and similarly for higher heights $h$ and appropriate type $h$ complexes in place of $M^\ell$); by the arguments above this can indeed not be equivalent to the infinite product. However, the arguments of \cite{heuts} require more technology than is necessary to include here.
\end{remark}
 
%\section*{Acknowledgements}
%The authors would like to thank the Hausdorff Institute for Mathematics in Bonn where the final stages of this paper were completed. The first author was partially supported by an ERP scholarship of the German Academic Scholarship Foundation. We would also like to thank the anonymous referee for pointing out a mistake in Lemma \ref{lem:filtrations}.%thanks to referee

\end{document}